\documentclass[10pt]{amsart}
\usepackage{color, graphicx}
\usepackage[square,sort,comma,numbers]{natbib}
\numberwithin{equation}{section} \overfullrule 5pt
\newtheorem{thm}{Theorem}[section]
\newtheorem{cor}[thm]{Corollary}
\newtheorem{lem}[thm]{Lemma}

\theoremstyle{definition}

\allowdisplaybreaks

\title{New hook-content formulas for strict partitions}
\author[Guo-Niu Han]{Guo-Niu Han}
\author[Huan Xiong]{Huan Xiong$^*$}
\address{Universit\'e de Strasbourg, CNRS, IRMA UMR 7501, F-67000 Strasbourg, France}
\email{guoniu.han@unistra.fr, \quad xiong@math.unistra.fr}



\thanks{An extended abstract for this paper appeared in the Proceedings of the FPSAC'16 conference.}

\thanks{$^*$ Huan Xiong is the corresponding author.}

\begin{document}
\begin{abstract} 
We introduce the difference operator for functions defined on strict partitions and prove a polynomiality property for a summation involving the hook length and content statistics. As an application, several new hook-content formulas for strict partitions are derived.
\end{abstract}
\maketitle

\section{Introduction} \label{sec:introduction}  


The basic knowledge on partitions,  Young
tableaux and symmetric functions could be found in \cite{ec2}.  In this paper, we focus on strict partitions.  A \emph{strict partition} is a finite strict decreasing sequence of
positive integers $\lambda = (\lambda_1, \lambda_2, \ldots,
\lambda_\ell)$. The integer  $|\lambda|=\sum_{1\leq i\leq
\ell}\lambda_i$ is called the \emph{size} of the
partition $\lambda$ and $\ell(\lambda)=\ell$ is called the {\it length} of $\lambda.$ For convenience, let $\lambda_i=0$ for $i>l(\lambda)$. A strict partition $\lambda$ can be identified with its shifted Young diagram, which means that the $i$-th row of the usual Young diagram is shifted to the right by $i$ boxes. Therefore the leftmost box in the $i$-th row has coordinate $(i,i+1)$. For the $(i, j)$-box in the shifted Young diagram of the strict partition $\lambda$, we can associate its \emph{hook length},
denoted by $h_{(i, j)}$, which is the number of boxes exactly to the
right, or  exactly above, or the box itself, plus~$\lambda_{j}$.
For example, consider the box $\square=(i,j)=(1,3)$ in the shifted Young diagram of the strict partition $(7,5,4,1)$.
There are 1 and 5 boxes above and to the right of the box $\square$ respectively. Since $\lambda_3=4$, the hook length of $\square$ is equal to $1+5+1+4=11$, as illustrated in Figure \ref{fig:1}.
The {\it content} of $\square=(i,j)$ is defined to be $c_\square=j-i$,
so that the leftmost box in each row has content $1$.
Also, 
let $\mathcal{H}(\lambda)$ be the multi-set of hook lengths of boxes
and
$H_{\lambda}$ be the product of all hook lengths of boxes in~$\lambda$.
\medskip

Our goal is to find some hook-content formulas for {\it strict partitions},
by analogy with that for {\it ordinary partitions}. 
For the ordinary partition $\nu$, it is well known that (see \cite{han2, rsk, ec2})
\begin{equation}
\label{eq:hookformula} f_\nu = \frac{|\nu| !}{H_{\nu}}
\qquad \text{and}\qquad \frac{1}{n!}\sum_{|\nu| =n}
f_\nu^2=1,
\end{equation}
where $H_{\nu}$ denotes the product of all hook lengths of boxes in $\nu$ and $f_{\nu}$ denotes the number of standard Young tableaux of shape
$\nu$.
The
first author conjectured \cite{han} that
$$
P(n)=\frac{1}{n!}\sum_{|\nu|= n}
f_\nu^2\sum_{\square\in\nu}h_{\square}^{2k} $$
is always a
polynomial in $n$ for all $k\in \mathbb{N}$, which was generalized and
proved by  Stanley \cite{stan}, and later generalized in \cite{hanxiong}  (see also   \cite{hanxiong2,han3,han4,hanxiong1,no,ols,ols2,panova}).

\begin{figure}
\centering
\begin{center}
\includegraphics[]{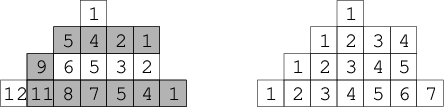}
\end{center}
\caption{The shifted Young diagram of the strict partition $(7,5,4,1)$ with its hook
lengths and contents.}\label{fig:1}
\end{figure}

\medskip

For two strict partitions $\lambda$ and $\mu$, we write $\lambda \supseteq \mu$
if $\lambda_i\geq \mu_i$ for all $i\geq 1$.
In this case,  the skew strict partition $\lambda/\mu$ can be identified with its
skew shifted Young diagram. For example, the skew strict partition $(7,5,4,1)/(4,2,1)$ is represented by the white boxes in Figure \ref{fig:2}.
Let $f_\lambda$ (resp.
$f_{\lambda/\mu}$) be the number of standard shifted Young tableaux of shape
$\lambda$ (resp. $\lambda/\mu$).  The following are well-known formulas   (see \cite{bandlow, schur,  thr}) analogous to \eqref{eq:hookformula}:
\begin{equation}
\label{eq:hookformula*} f_\lambda = \frac{|\lambda| !}{H_{\lambda}}
\qquad \text{and}\qquad \frac{1}{n!}\sum_{|\lambda| =n}
2^{n-\ell(\lambda)}f_\lambda^2=1.
\end{equation}

\begin{figure}
\centering
\begin{center}
\includegraphics[]{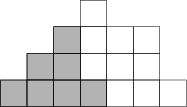}
\end{center}
\caption{The skew shifted Young diagram of the skew strict partition $(7,5,4,1)/(4,2,1)$.}\label{fig:2}
\end{figure}


In this paper, we generalize the latter equality of \eqref{eq:hookformula*}
by means of the following results.

\begin{thm} \label{th:polynomial:noq}
Suppose that $Q$ is a given symmetric function, and $\mu$ is a given strict partition. Then
\begin{align*}
P(n)=\sum_{|\lambda/\mu|=n}\frac{2^{|\lambda|-|\mu|-\ell(\lambda)+\ell(\mu)}f_{\lambda/\mu}}{H_{\lambda}}Q\left(\binom{c_{\square}}{2}:
{\square}\in\lambda\right)
\end{align*}
is a polynomial in $n$, where $Q(\binom{c_{\square}}{2}: {\square}\in\lambda)$
 means that $|\lambda|$ of
 the variables are substituted by $\binom{c_\square}{2}$ 
for $\square\in \lambda$, and all other variables by~$0$.
\end{thm}

\begin{thm} \label{th:polynomial*}
Suppose that $k$ is a given nonnegative integer. Then
\begin{align*}
\sum_{|\lambda|=n}\frac{2^{|\lambda|-\ell(\lambda)}f_{\lambda}}{H_{\lambda}}\sum_{\square\in\lambda}
\binom{c_{\square}+k-1}{2k}=\frac{2^k}{(k+1)!}\binom{n}{k+1}.
\end{align*}
\end{thm}

When $k=0$ we derive the latter identity of \eqref{eq:hookformula*}. When $k=1$, Theorem \ref{th:polynomial*} becomes 
\begin{align*}
\sum_{|\lambda|=n}\frac{2^{|\lambda|-\ell(\lambda)}f_{\lambda}}{H_{\lambda}}\sum_{\square\in\lambda}
\binom{c_{\square}}{2}=\binom{n}{2},
\end{align*}
which could also be obtained by setting $\mu=\emptyset$ in the next theorem. 

\begin{thm} \label{th:k=1}
Let $\mu$ be a strict partition. Then
\begin{equation}\label{eq:sumc:k=1}
\sum_{|\lambda/\mu|=n}\frac{2^{|\lambda|-\ell(\lambda)-|\mu|+\ell(\mu)}f_{\lambda/\mu}H_\mu}{H_{\lambda}}\left(\sum_{\square\in\lambda}\binom{c_{\square}}{2}-\sum_{\square\in\mu}
\binom{c_{\square}}{2}\right)
=\binom{n}{2}+n|\mu|.
\end{equation}
\end{thm}

The proofs of those theorems are given in Section 4, by using the difference operator technique.


\section{Difference operators} 
For each strict partition $\lambda$, the symbol $\lambda^+$ (resp. $\lambda^-$)
always represents a strict partition obtained by adding (resp. removing) a box to (resp. from) $\lambda$. In other words, $|\lambda^+/\lambda|=1$ and 
$|\lambda/\lambda^-|=1$.
By analogy with the difference operator for ordinary partitions introduced in \cite{hanxiong}, we define
{\it the difference operator for strict partitions}  by
$$
D\left(g(\lambda)\right):=\sum\limits_{\ell(\lambda^+)>\ell(\lambda) }g(\lambda^+)+2\sum\limits_{\ell(\lambda^+)=\ell(\lambda)}g(\lambda^+)\ -\ g(\lambda),
$$
where $\lambda$ and $\lambda^+$ are strict partitions and $g$ is a function on strict partitions. 
Notice that
$ \#\{\lambda^+:\ell(\lambda^+)>\ell(\lambda) \}=0\ \text{or} \ 1.$

\medskip

For each skew strict partition $\lambda/\mu$, let
${f'}_{\lambda/\mu}:=2^{|\lambda|-|\mu|-\ell(\lambda)+\ell(\mu)}f_{\lambda/\mu}.$

\begin{lem} \label{th:f'}
For two different strict partitions $\lambda\supseteq\mu$ we have
\begin{align*}
{f'}_{\lambda/\mu}=\sum\limits_{\lambda^-:\substack{\lambda\supseteq
\lambda^-\supseteq \mu \\ 
\ell(\lambda^-)<\ell(\lambda) }}{f'}_{\lambda^-/\mu}
+2\sum\limits_{\lambda^-:\substack{\lambda\supseteq
\lambda^-\supseteq \mu \\ 
\ell(\lambda^-)=\ell(\lambda) }}{f'}_{\lambda^-/\mu}.
\end{align*}
\end{lem}
\begin{proof}
By the construction of standard shifted Young tableaux we have
\begin{align*}
{f}_{\lambda/\mu}=\sum\limits_{\lambda\supseteq
\lambda^-\supseteq \mu }{f}_{\lambda^-/\mu}
\end{align*}
and therefore 
\begin{align*}
2^{|\lambda|-\ell(\lambda)}{f}_{\lambda/\mu}=\sum\limits_{\lambda^-:\substack{\lambda\supseteq
\lambda^-\supseteq \mu \\ 
\ell(\lambda^-)<\ell(\lambda) }}2^{|\lambda^-|-\ell(\lambda^-)}{f}_{\lambda^-/\mu}
+2\sum\limits_{\lambda^-:\substack{\lambda\supseteq
\lambda^-\supseteq \mu \\ 
\ell(\lambda^-)=\ell(\lambda) }}2^{|\lambda^-|-\ell(\lambda^-)}{f}_{\lambda^-/\mu}.
\end{align*}
Then by the definition of ${f'}_{\lambda/\mu}$ we prove the claim.
\end{proof}

\begin{lem} \label{th:telescope*}
For each strict partition $\mu$ and each function  $g$ of strict
partitions, let
\begin{align*}
A(n)&:=\sum_{|\lambda/\mu|=n}{f'}_{\lambda/\mu}g(\lambda)\\
\noalign{\noindent and}
B(n)&:=\sum_{|\lambda/\mu|=n}{f'}_{\lambda/\mu}Dg(\lambda).\\
\noalign{\noindent Then} A(n)&=A(0)+\sum_{k=0}^{n-1}B(k).
\end{align*}
\end{lem}
\begin{proof}
We have
\begin{align*}
A(n+1)-A(n)\
&=\sum_{|\gamma/\mu|=n+1}{f'}_{\gamma/\mu}g(\gamma)-\sum_{|\lambda/\mu|=n}{f'}_{\lambda/\mu}g(\lambda)
\\&=\sum_{|\gamma/\mu|=n+1}\left(\sum\limits_{\gamma^-:\substack{\gamma\supseteq
\gamma^-\supseteq \mu \\ 
\ell(\gamma^-)<\ell(\gamma) }}{f'}_{\gamma^-/\mu}
+2\sum\limits_{\gamma^-:\substack{\gamma\supseteq
\gamma^-\supseteq \mu \\ 
\ell(\gamma^-)=\ell(\gamma)
}}{f'}_{\gamma^-/\mu}\right)g(\gamma)\\&\ \ \ \ -\sum_{|\lambda/\mu|=n}{f'}_{\lambda/\mu}g(\lambda)
\\&=\sum_{|\lambda/\mu|=n}{f'}_{\lambda/\mu}
\left(\sum\limits_{\ell(\lambda^+)>\ell(\lambda)}g(\lambda^+)+2\sum\limits_{\ell(\lambda^+)=\ell(\lambda)}g(\lambda^+)-g(\lambda)\right)
\\&=\sum_{|\lambda/\mu|=n}{f'}_{\lambda/\mu}Dg(\lambda)=B(n).
\end{align*}
Thus
\begin{align*}
    A(n+1)&=A(n)+B(n)\\
                &= A(n-1)+B(n-1)+B(n)\\
                &=\cdots \\
                &= A(0)+\sum_{k=0}^{n}B(k).\qedhere
\end{align*}
\end{proof}

\begin{thm}\label{th:telescope}
Let $g$ be a function on strict partitions and $\mu$ be a given
strict partition. Then we have
\begin{equation}\label{eq:telescope*}
\sum_{
|\lambda/\mu|=n}
2^{|\lambda|-|\mu|-\ell(\lambda)+\ell(\mu)}f_{\lambda/\mu}g(\lambda)
=\sum_{k=0}^n\binom{n}{k}D^kg(\mu)
\end{equation}
and
\begin{equation}\label{eq:telescope**}
D^ng(\mu)=\sum_{k=0}^n(-1)^{n+k}\binom{n}{k}
\sum_{|\lambda/\mu|=k}2^{|\lambda|-|\mu|-\ell(\lambda)+\ell(\mu)}f_{\lambda/\mu}g(\lambda).
\end{equation}
In particular, if there exists some positive integer $r$ such that
$D^r g(\lambda)=0$ for every strict partition $\lambda$, then the
left-hand side of \eqref{eq:telescope*} is a polynomial in $n$ with
degree at most $r-1$.
\end{thm}

\begin{proof}
We will prove \eqref{eq:telescope*} by induction. The case $n=0$ is
trivial. Assume that \eqref{eq:telescope*} is true for some
nonnegative integer $n$. Then by Lemma \ref{th:telescope*} we have
\begin{align*}
\sum_{ |\lambda/\mu|=n+1}{f'}_{\lambda/\mu}g(\lambda)&=
\sum_{ |\lambda/\mu|=n}{f'}_{\lambda/\mu}g(\lambda)+
\sum_{ |\lambda/\mu|=n}{f'}_{\lambda/\mu}Dg(\lambda)
\\&=\sum_{k=0}^n\binom{n}{k}D^kg(\mu)+\sum_{k=0}^n\binom{n}{k}D^{k+1}g(\mu)
\\&=\sum_{k=0}^{n+1}\binom{n+1}{k}D^kg(\mu).
\end{align*}

Identity \eqref{eq:telescope**} follows from the M\"obius inversion
formula \cite{Rota1964}.
\end{proof}

{\it Example}. Let $g(\lambda)=1/H_\lambda$. Then $Dg(\lambda)=0$ by
Theorem \ref{th:sumhlambda}. The two quantities defined in Lemma
\ref{th:telescope*} are:
$$A(n)=\sum_{|\lambda/\mu|=n}  \frac{{{f'}_{\lambda/ \mu}}}{H_{\lambda}}
\text{\qquad and\qquad} B(n)=0.$$ Consequently,
\begin{equation}\label{eq:skewhook}
\sum_{|\lambda/\mu|=n}
\frac{{2^{|\lambda|-|\mu|-\ell(\lambda)+\ell(\mu)}f_{\lambda/\mu}}}{H_{\lambda}}
= \frac 1{H_\mu}.
\end{equation}
In particular,   $\mu=\emptyset$ implies
\begin{equation}\label{eq:skewhook**}
\sum_{|\lambda|=n}
\frac{{2^{|\lambda|-\ell(\lambda)}f_{\lambda}}}{H_{\lambda}} = 1,
\end{equation}
or equivalently,
\begin{equation}\label{eq:skewhook***}
\sum_{|\lambda|=n} {2^{|\lambda|-\ell(\lambda)}f_{\lambda}^2} = n!.
\end{equation}


\section{Corners of strict partitions} 

For a strict partition $\lambda$, the \emph{outer corners} are the
boxes which can be removed to get a new strict partition
$\lambda^-$. Let $(\alpha_1,\beta_1),\ldots,(\alpha_{m},\beta_{m})$
be the coordinates of outer corners such that
$\alpha_1>\alpha_2>\cdots
>\alpha_m$. Let $y_j=\beta_j-\alpha_j$ be the contents of outer
corners for $1\leq j \leq m.$ We set $\alpha_{m+1}=0,\
\beta_0=\ell(\lambda)+1$ and call
$(\alpha_1,\beta_0),(\alpha_2,\beta_1),\ldots,(\alpha_{m+1},\beta_{m})$
 the  \emph{inner corners} of $\lambda$. Let $x_i=\beta_i-\alpha_{i+1}$
 be the contents of inner corners for $0\leq i \leq m$ (see Figure \ref{fig:3}).
The following relations of $x_i$ and $y_j$ are obvious.
\begin{equation}
x_0=1\leq y_1<x_1<y_2<x_2<\cdots <y_m <x_m.
\end{equation}
Notice that $y_1=1$ iff $\lambda_{\ell(\lambda)}=1$.

We define
\begin{equation}\label{def:qk}
q_k(\lambda):=\sum_{i=0}^m \binom{x_i}{2}^{k}-\sum_{i=1}^m
\binom{y_i}{2}^{k}
\end{equation}
for each $k\geq 0$.
For each partition $\nu=(\nu_1, \nu_2, \ldots, \nu_\ell)$ we define
the function  $q_\nu(\lambda)$ of strict partitions by
\begin{equation}
q_\nu(\lambda):=q_{\nu_1}(\lambda)q_{\nu_2}(\lambda)\cdots
q_{\nu_\ell}(\lambda).
\end{equation}

First we consider the difference between the hook length sets of $\lambda$ and $\lambda^+=\lambda\cup \{ \Box\}$ for some box $\Box$.

\begin{thm} \label{th:Hlambda}
Suppose that $\lambda^+=\lambda\cup \{ \square\}$ such that
$c_\square=x_i.$ If $i=0,$ then
\begin{equation*}
\frac{H_{\lambda}}{H_{{\lambda^+}}}=\frac{\prod\limits_{\substack{1\leq
j\leq
m}}\left(\binom{x_0}{2}-\binom{y_j}{2}\right)}{\prod\limits_{\substack{1\leq
j\leq m}}\left(\binom{x_0}{2}-\binom{x_j}{2}\right)}.
\end{equation*} If $1\leq
i\leq m,$ then
\begin{equation*}
\frac{H_{\lambda}}{H_{{\lambda^+}}}=\frac12\cdot\frac{\prod\limits_{\substack{1\leq
j\leq m}}\left(\binom{x_i}{2}-\binom{y_j}{2}\right)}{\prod\limits_{\substack{0\leq j\leq m\\
j\neq i}}\left(\binom{x_i}{2}-\binom{x_j}{2}\right)}.
\end{equation*}
\end{thm}

\begin{figure}[!b] 
\centering
\begin{center}
\includegraphics[]{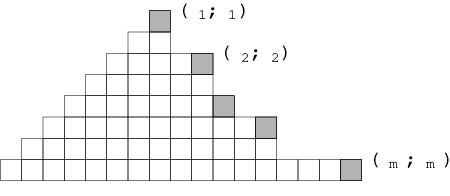}
\end{center}
\caption{A strict partition and its corners. The outer corners are labelled with
$(\alpha_i, \beta_i)$ ($i=1,2,\ldots, m$). The inner corners are indicated by
the dot symbol ``$\cdot$''.}\label{fig:3}
\end{figure}

\begin{proof}
First we consider the case $i=0$, which means that $y_1\geq 2$.  In this case we add the box $\square=(\ell+1,\ell+2)$ to $\lambda$ with $c_\square=x_0=1$ where $\ell$ is the length of $\lambda$. By the definition, it is easy to see that the hook lengths of boxes
which are in the $(\ell+1)$-th column and $(\ell+2)$-th column of $\lambda$
 increase by $1$, and the hook lengths of boxes in the other columns don't change. Since the boxes
which are in the $\ell+1$-column and $\ell+2$-column of $\lambda$ have hook lengths 
\begin{align*}
    \bigcup_{j=1}^{m}\left(\{h:y_j-1\leq h\leq x_{j}-2\}\cup \{h:y_j\leq h\leq x_{j}-1\} \right),
\end{align*}
then the boxes
which are in the $\ell+1$-column and $\ell+2$-column of $\lambda^+$ have hook lengths 
\begin{align*}
   \{1\} \cup \left(\bigcup_{j=1}^{m}\left(\{h:y_j\leq h\leq x_{j}-1\}\cup \{h:y_j+1\leq h\leq x_{j}\} \right)\right).
\end{align*}
 Therefore
\begin{align*}\mathcal{H}(\lambda)\setminus \mathcal{H}(\lambda^+)\ &=
\{ y_1,y_1-1,y_2,y_2-1,\cdots, y_m,y_m-1 \}
\\ &\qquad \setminus 
\{ 1,x_1,x_1-1,x_2,x_2-1,\cdots, x_m,x_m-1 \}
,
\end{align*}
which means that
\begin{equation*}
\frac{H_{\lambda}}{H_{{\lambda^+}}}=\frac{\prod\limits_{\substack{1\leq
j\leq m}}y_j(y_j-1)}{\prod\limits_{\substack{1\leq j\leq
m}}x_j(x_j-1)}=\frac{\prod\limits_{\substack{1\leq j\leq
m}}\left(\binom{x_0}{2}-\binom{y_j}{2}\right)}{\prod\limits_{\substack{1\leq
j\leq m}}\left(\binom{x_0}{2}-\binom{x_j}{2}\right)}
\end{equation*}
since $x_0=1$.

Similarly, for the case $1\leq i\leq m$, we add
the box $\square=(\alpha_{i+1}+1,\beta_i+1)$ to $\lambda$ with $c_\square=x_i$.  By the definition, it is easy to see that the hook lengths of boxes
which are in the $(\alpha_{i+1}+1)$-th column, $(\beta_i+1)$-th column and $(\alpha_{i+1}+1)$-th row of $\lambda$
 increase by $1$, and the hook lengths of other boxes don't change. Since the boxes
which are in the $(\alpha_{i+1}+1)$-th column, $(\beta_i+1)$-th column and $(\alpha_{i+1}+1)$-th row of $\lambda$ have hook lengths 
\begin{align*}
  \bigcup_{j=i+1}^{m}\{h:x_i+y_j-1\leq h\leq x_{i}+x_j-2\},
\end{align*}
\begin{align*}
  \bigcup_{j=i+1}^{m}\{h:y_j-x_i\leq h\leq x_{j}-x_i-1\},
\end{align*}
and
\begin{align*}
  & \{h:x_i-y_1\leq h\leq x_{i}-1\} \cup \{h:x_i+y_i-1\leq h\leq 2x_{i}-3\}  \\&  \cup \bigcup_{j=1}^{i-1}\left(\{h:x_i-y_{j+1}\leq h\leq x_{i}-x_j-1\}
   \cup \{h:x_i+y_{j}-1\leq h\leq x_{i}+x_j-2\}\right)
\end{align*}
respectively,
then the same boxes
which are in the $(\alpha_{i+1}+1)$-th column, $(\beta_i+1)$-th column and $(\alpha_{i+1}+1)$-th row of $\lambda^+$ have hook lengths 
\begin{align*}
  \bigcup_{j=i+1}^{m}\{h:x_i+y_j\leq h\leq x_{i}+x_j-1\},
\end{align*}
\begin{align*}
  \bigcup_{j=i+1}^{m}\{h:y_j-x_i+1\leq h\leq x_{j}-x_i\},
\end{align*}
and
\begin{align*}
  & \{h:x_i-y_1+1\leq h\leq x_{i}\} \cup \{h:x_i+y_i\leq h\leq 2x_{i}-2\}  \\&  \cup \bigcup_{j=1}^{i-1}\left(\{h:x_i-y_{j+1}+1\leq h\leq x_{i}-x_j\}
   \cup \{h:x_i+y_{j}\leq h\leq x_{i}+x_j-1\}\right)
\end{align*}
respectively.
Notice that the box $\square=(\alpha_{i+1}+1,\beta_i+1)$ in $\lambda^+$ has hook length $1$,
then we have
\begin{align*}
\mathcal{H}(\lambda)\setminus \mathcal{H}(\lambda^+)
&=
\left( \{ |x_i-y_j|:1\leq j \leq m \} \cup \{ x_i+y_j-1:1\leq j \leq m \} \right)\\
&\qquad \setminus 
\bigl( \{ 1,x_i, 2x_i-2 \} \cup  
\{ |x_i-x_j|:1\leq j \leq m, j\neq i \} \\ 
& \qquad\qquad \cup \{x_i+x_j-1:1\leq j \leq m, j\neq i \} \bigr)
,
\end{align*}
which means that  
\begin{equation*}
\frac{H_{\lambda}}{H_{{\lambda^+}}}=\frac{1}{x_i(2x_i-2)}\cdot\frac{\prod\limits_{\substack{1\leq
j\leq m}}(x_i-y_j)(x_i+y_j-1)}{\prod\limits_{\substack{1\leq j\leq m\\
j\neq
i}}(x_i-x_j)(x_i+x_j-1)}=\frac12\cdot\frac{\prod\limits_{\substack{1\leq
j\leq m}}\left(\binom{x_i}{2}-\binom{y_j}{2}\right)}{\prod\limits_{\substack{0\leq j\leq m\\
j\neq i}}\left(\binom{x_i}{2}-\binom{x_j}{2}\right)}.\qedhere
\end{equation*}
\end{proof} 

\medskip
Suppose that $a_0<a_1<\cdots<a_m$ and ${b}_1<\cdots<{b}_m$
are real numbers. Let
\begin{equation}\label{def:qk*} q_k(\{a_i\},\{{b}_i\}):=\sum_{i=0}^m
{a_i}^{k}-\sum_{i=1}^m {{b}_i}^{k}
\end{equation}
for each $k\geq 0$ and
\begin{equation}
q_\nu(\{a_i\},\{{b}_i\}):=\prod_{j=1}^{\ell}q_{\nu_j}(\{a_i\},\{{b}_i\})
\end{equation}
for the usual partition $\nu=(\nu_1, \nu_2, \ldots, \nu_\ell)$.
Notice
that $$q_k(\lambda)=q_k\left(\{\binom{x_i}{2}\}_{0\leq i\leq
m},\{\binom{y_i}{2}\}_{1\leq i\leq m}\right).$$
\begin{thm} \label{th:aibi}
Let $k$ be a nonnegative  integer. Then there exist some
$\xi_{\nu}\in \mathbb{Q}$ such that
\begin{equation*} \sum_{0\leq i\leq m}
\frac{\prod\limits_{\substack{1\leq
j\leq m}}({a_i}-{{b}_j})}{\prod\limits_{\substack{0\leq j\leq m\\
j\neq i}}({a_i}-{a_j})}a_i^k
=\sum_{|\nu|\leq k}\xi_{\nu}q_\nu(\{a_i\},\{{b}_i\})
\end{equation*}
for all  real numbers $a_0<a_1<\cdots<a_m$ and
${b}_1<{b}_2<\cdots<{b}_m$. 
\end{thm}

\begin{proof}
First notice we just need to prove the case that $a_i\neq 0$ for all $i$. Because if we multiply by $\prod_{0\leq i< j \leq n}({a_i}-{a_j})$ on both sides of the above formula, then both sides become polynomials in $a_0,a_1,\ldots,a_m$ and
${b}_1,{b}_2,\ldots,{b}_m$, which means they are continuous functions on such variables. Therefore if the above formula is true for all nonzero $a_i$, then it is also true for the case $a_i= 0$ for some $i$. 
Let
$$
g(z)=\prod_{1\leq j\leq m}(1-{b}_jz) -
\sum_{0\leq i\leq m}\frac{\prod\limits_{\substack{1\leq
j\leq m}}({a_i}-{{b}_j})}{\prod\limits_{\substack{0\leq j\leq m\\
j\neq i}}({a_i}-{a_j})} \prod_{\substack{0\leq j\leq m\\ j\neq
i}}(1-{a}_j z).
$$
Then for $0\leq i\leq m$  we obtain
\begin{align*}
    g\left(\frac{1}{a_i}\right)&= \prod_{1\leq j\leq m}\left(1-\frac{{b}_j}{a_i}\right)-
\frac{\prod\limits_{\substack{1\leq
j\leq m}}({a_i}-{{b}_j})}{\prod\limits_{\substack{0\leq j\leq m\\
j\neq i}}({a_i}-{a_j})} \prod\limits_{\substack{0\leq j\leq m\\
j\neq i}}\left(1-\frac{a_j}{a_i}\right) = 0.
\end{align*}
This means that $g(z)$ has at least $m+1$ roots, so that $g(z)=0$
since $g(z)$ is a polynomial in $z$ with degree at most $m$.
Therefore
$$
\sum_{0\leq i\leq m}\frac{\prod\limits_{\substack{1\leq
j\leq m}}({a_i}-{{b}_j})}{\prod\limits_{\substack{0\leq j\leq m\\
j\neq i}}({a_i}-{a_j})}\cdot \frac{1}{1-a_iz} =\frac{\prod_{1\leq
j\leq m}(1-{b}_jz)}{\prod_{0\leq j\leq m}(1-a_j z)},
$$
which means that
\begin{align}
&\sum_{0\leq i\leq m}\frac{\prod\limits_{\substack{1\leq
j\leq m}}({a_i}-{{b}_j})}{\prod\limits_{\substack{0\leq j\leq m\\
j\neq i}}({a_i}-{a_j})}\left(\sum_{k\geq0}(a_iz)^k\right) \label{eq:ab}
\\ 
&=
\exp\left(\sum_{1\leq j\leq m}\ln(1-{b}_jz)- \sum_{0\leq i\leq
m}\ln(1-a_iz)\right)\nonumber
\\
&= \exp\left(\sum_{k\geq
1}\frac{q_k(\{a_i\},\{{b}_i\})}{k}z^k\right).\nonumber 
\end{align}
Comparing the coefficients of $z^k$ on both sides, we obtain there
exist some $\xi_{\nu}\in \mathbb{Q}$ such that
\begin{equation*} \sum_{0\leq i\leq m}
\frac{\prod\limits_{\substack{1\leq
j\leq m}}({a_i}-{{b}_j})}{\prod\limits_{\substack{0\leq j\leq m\\
j\neq i}}({a_i}-{a_j})}a_i^k=\sum_{|\nu|\leq
k}\xi_{\nu}q_\nu(\{a_i\},\{{b}_i\})
\end{equation*}  for all real
numbers $a_0<a_1<\cdots<a_m$ and ${b}_1<{b}_2<\cdots<{b}_m$.
\end{proof}

By \eqref{eq:ab}, when $k=0,1,2$, we obtain
\begin{align} 
\sum_{0\leq i\leq m}
\frac{\prod\limits_{\substack{1\leq
j\leq m}}({a_i}-{{b}_j})}{\prod\limits_{\substack{0\leq j\leq m\\
j\neq i}}({a_i}-{a_j})} &=1,\label{eq:ab0} \\
\sum_{0\leq i\leq m}
\frac{\prod\limits_{\substack{1\leq
j\leq m}}({a_i}-{{b}_j})}{\prod\limits_{\substack{0\leq j\leq m\\
j\neq i}}({a_i}-{a_j})} a_i&=q_1(\{a_i\}, \{b_i\}), \label{eq:ab1}\\
\sum_{0\leq i\leq m}
\frac{\prod\limits_{\substack{1\leq
j\leq m}}({a_i}-{{b}_j})}{\prod\limits_{\substack{0\leq j\leq m\\
j\neq i}}({a_i}-{a_j})} a_i^2&=\frac {q_1^2(\{a_i\}, \{b_i\}) + q_2(\{a_i\}, \{b_i\})}{2}. \label{eq:ab2} 
\end{align}

Let $\lambda^{i+}=\lambda\cup \{ \square_i\}$ such that
$c_{\square_i}=x_i$ for $1\leq i\leq m$. If $y_1> 1,$ let
$\lambda^{0+}=\lambda\cup \{ \square_0\}$ such that
$c_{\square_0}=x_0=1.$

\begin{thm} \label{th:sumhlambda}
Suppose that $\lambda$ is a given strict partition.  Then
\begin{equation*} 
	D\left(\frac{1}{H_{{\lambda}}}\right) =0.
\end{equation*}
\end{thm}
\begin{proof}
 Notice that when $y_1=1,$
we have $\{\lambda^+:\ell(\lambda^+)>\ell(\lambda) \}=\emptyset,$
therefore
$$
\sum\limits_{{\ell(\lambda^+)>\ell(\lambda)}}\frac{H_{\lambda}}{H_{{\lambda^{+}}}}=0=\frac{\prod\limits_{\substack{1\leq
j\leq
m}}\left(\binom{x_0}{2}-\binom{y_j}{2}\right)}{\prod\limits_{\substack{1\leq
j\leq m}}\left(\binom{x_0}{2}-\binom{x_j}{2}\right)}.
$$
When $y_1>1,$ we have
$\{\lambda^+:\ell(\lambda^+)>\ell(\lambda)\}=\{\lambda^{0+}\},$
therefore by Theorem \ref{th:Hlambda} we also obtain
$$
\sum\limits_{{\ell(\lambda^+)>\ell(\lambda)}}\frac{H_{\lambda}}{H_{{\lambda^{+}}}}=\frac{\prod\limits_{\substack{1\leq
j\leq
m}}\left(\binom{x_0}{2}-\binom{y_j}{2}\right)}{\prod\limits_{\substack{1\leq
j\leq m}}\left(\binom{x_0}{2}-\binom{x_j}{2}\right)}
$$
and
\begin{align*}
\sum\limits_{{\ell(\lambda^+)>\ell(\lambda)}}\frac{H_{\lambda}}{H_{{\lambda^+}}}
+2\sum\limits_{{\ell(\lambda^+)=\ell(\lambda)}}\frac{H_{\lambda}}{H_{{\lambda^+}}}
=\sum_{0\leq i\leq m} \frac{\prod\limits_{\substack{1\leq
j\leq m}}\left({\binom{x_i}{2}}-{\binom{y_j}{2}}\right)}{\prod\limits_{\substack{0\leq j\leq m\\
j\neq i}}\left({\binom{x_i}{2}}-{\binom{x_j}{2}}\right)}.
\end{align*}
Let $a_i=\binom{x_i}{2}$ and $b_i=\binom{y_i}{2}$ in \eqref{eq:ab0},
we obtain 
\begin{equation*}
D\left(\frac{1}{H_{{\lambda}}}\right)=\frac{1}{H_{{\lambda}}}\left(
\sum\limits_{{\ell(\lambda^+)>\ell(\lambda)}}\frac{H_{\lambda}}{H_{{\lambda^{+}}}}
+2\sum\limits_{{\ell(\lambda^+)=\ell(\lambda)}}\frac{H_{\lambda}}{H_{{\lambda^+}}}-1\right)=0.\qedhere
\end{equation*}
\end{proof}

\begin{cor} \label{th:sumhlambda*}
Suppose that $g$ is a function on strict partitions.  Then
\begin{align*}
D\left(\frac{g(\lambda)}{H_{{\lambda}}}\right)=
\sum\limits_{{\ell(\lambda^+)>\ell(\lambda)}}\frac{g(\lambda^+)-g(\lambda)}{H_{{\lambda^{+}}}}
+2\sum\limits_{{\ell(\lambda^+)=\ell(\lambda)}}\frac{g(\lambda^+)-g(\lambda)}{H_{{\lambda^+}}}
\end{align*}
for every strict partition $\lambda.$
\end{cor}
\begin{proof}
The corollary follows directly from the definition of the operator $D$ and the last identity in the proof of Theorem \ref{th:sumhlambda}.
\end{proof}

\begin{thm} \label{th:diffq1}
Let $k$ be a given nonnegative integer and $\lambda$ be a strict
partition. Then
 there exist some $\xi_j\in \mathbb{Q}$ such that
$$
q_k(\lambda^{i+})-q_k(\lambda)=\sum_{j=0}^{k-1}\xi_j{\binom{x_i}{2}}^j
$$
for every strict partition $\lambda$ and every $i$.
\end{thm}
\begin{proof}
Denote by $X=\{x_0,x_1,\ldots, x_m\}$ and $Y=\{y_1, y_2, \ldots,
y_m\}$.
For $1\leq i\leq m$, four cases are to be considered. 

(i) If
$\beta_{i}+1<\beta_{i+1}$ and $\alpha_{i+1}+1<\alpha_{i}$. Then it
is easy to see that the contents of inner corners and outer corners
of $\lambda^{i+}$ are $X\cup \{x_i-1, x_i+1\} \setminus \{x_i\}$ and
$Y\cup \{x_i\}$ respectively. 

(ii) If $\beta_{i}+1=\beta_{i+1}$ and
$\alpha_{i+1}+1<\alpha_{i}$, so that $y_{i+1}=x_i+1$. Hence the
contents of inner corners and outer corners of $\lambda^{i+}$ are
$X\cup \{x_i-1\} \setminus \{x_i\}$ and $Y\cup \{x_i\}\setminus
\{x_i+1\}$ respectively. 

(iii) If $\beta_{i}+1<\beta_{i+1}$ and
$\alpha_{i+1}+1=\alpha_{i}$, so that $y_{i}=x_i-1$. Then the
contents of inner corners and outer corners of $\lambda^{i+}$ are
$X\cup \{x_i+1\} \setminus \{x_i\}$ and $Y\cup \{x_i\}\setminus
\{x_i-1\}$ respectively. 

(iv) If $\beta_{i}+1=\beta_{i+1}$ and
$\alpha_{i+1}+1=\alpha_{i}$. Then $y_{i}+1=x_i=y_{i+1}-1$. The
contents of inner corners and outer corners of $\lambda^{i+}$ are $X
\setminus \{x_i\}$ and $Y\cup \{x_i\}\setminus \{x_i-1, x_i+1\}$
respectively.

For $i=0$, two cases are to be considered.  

(v) If $y_1=2,$ the
contents of inner corners and outer corners of $\lambda^{0+}$ are
$X$ and $Y\cup \{1\}\setminus \{2\}$ respectively. 

(vi) If $y_1>2,$
the contents of inner corners and outer corners of $\lambda^{i+}$
are $X\cup\{2\}$ and $Y\cup \{1\}$ respectively.

In each of the six cases, we always have
\begin{equation}
q_k(\lambda^{i+})-q_k(\lambda)={\binom{x_i+1}{2}}^k+{\binom{x_i-1}{2}}^k
-2{\binom{x_i}{2}}^k.
\end{equation}

Next we have for all $z\in \mathbb{R}$,
\begin{equation*}(z+2)^{2k}+(z-2)^{2k}-2z^{2k}=2{\sum_{1\leq j
\leq {k}}\binom{2k}{2j}2^{2j}{z}^{2k-2j}}.
\end{equation*}
Replace $z$ by $2z-1,$ we obtain
\begin{equation*}
(2z+1)^{2k}+(2z-3)^{2k}-2(2z-1)^{2k}=2{\sum_{1\leq j \leq
{k}}\binom{2k}{2j}2^{2j}{(2z-1)}^{2k-2j}},
\end{equation*}
or
\begin{align*}
	&\left(8\binom{z+1}{2}+1\right)^{k}+\left(8\binom{z-1}{2}+1\right)^{k}-2\left(8\binom{z}{2}+1\right)^{k} \\
 = &2{\sum_{1\leq j \leq
{k}}\binom{2k}{2j}2^{2j}{\left(8\binom{z}{2}+1\right)}^{k-j}}.
\end{align*}
Then by induction on $k$ we have
\begin{equation*}
{\binom{x_i+1}{2}}^k+{\binom{x_i-1}{2}}^k
-2{\binom{x_i}{2}}^k=\sum_{j=0}^{k-1}\xi_j{\binom{x_i}{2}}^j
\end{equation*}
for some constants $\xi_j\in \mathbb{Q}$. 
\end{proof}

\begin{thm} \label{th:diffq2}
Let $\nu=(\nu_1, \nu_2, \ldots, \nu_\ell)$ be a partition. Then
there exist some $\xi_{\delta}\in \mathbb{Q}$ such that
\begin{equation}\label{eq:Dnu}
D\left(\frac{q_{\nu}(\lambda)}{H_{\lambda}}\right)= \sum_{|\delta|\leq
|\nu|-1}\xi_{\delta}\frac{q_\delta(\lambda)}{H_{\lambda}}
\end{equation}
for every strict partition $\lambda$.
\end{thm}
\begin{proof} 
For $0\leq i \leq m$, we have
\begin{equation}\label{eq:delta:q}
\prod_{k=1}^{\ell}q_{\nu_k}(\lambda^{i+})-\prod_{k=1}^{\ell}q_{\nu_k}(\lambda)
=   
\sum\limits_{(*)}\prod_{k\in U}q_{\nu_k}(\lambda)
{\prod_{k'\in V} \left(q_{\nu_{k'}}(\lambda^{i+})-q_{\nu_{k'}}(\lambda)\right)},
\end{equation}
where the sum $(*)$ ranges over all
pairs $(U,V)$ of positive integer sets such that $U\cup V=\{1,2,\ldots,\ell\},\,  U\cap
V=\emptyset$ and $V\neq \emptyset$. Actually the Identity \eqref{eq:delta:q} follows by the inclusion-exclusion principle.
By Corollary \ref{th:sumhlambda*} and Theorem \ref{th:Hlambda} we have 
\begin{align*}
&H_{\lambda}D\bigl(\frac{q_{\nu}(\lambda)}{H_{{\lambda}}}\bigr)
=
\sum\limits_{{\ell(\lambda^+)>\ell(\lambda)}}\frac{H_{\lambda}\left(q_{\nu}(\lambda^+)-q_{\nu}(\lambda)\right)}{H_{{\lambda^{+}}}}
+2\sum\limits_{{\ell(\lambda^+)=\ell(\lambda)}}\frac{H_{\lambda}\left(q_{\nu}(\lambda^{+})-q_{\nu}(\lambda)\right)}{H_{{\lambda^+}}}\\
&=
\sum_{0\leq i\leq m} \frac{\prod\limits_{\substack{1\leq
j\leq m}}\left({\binom{x_i}{2}}-{\binom{y_j}{2}}\right)}{\prod\limits_{\substack{0\leq j\leq m\\
j\neq i}}\left({\binom{x_i}{2}}-{\binom{x_j}{2}}\right)}\left(\prod_{k=1}^{\ell}q_{\nu_k}(\lambda^{i+})-\prod_{k=1}^{\ell}q_{\nu_k}(\lambda)\right)\\
&=   
\sum\limits_{(*)}\prod_{k\in U}q_{\nu_k}(\lambda)\sum_{0\leq i\leq m} \frac{\prod\limits_{\substack{1\leq
j\leq m}}\left({\binom{x_i}{2}}-{\binom{y_j}{2}}\right)}{\prod\limits_{\substack{0\leq j\leq m\\
j\neq i}}\left({\binom{x_i}{2}}-{\binom{x_j}{2}}\right)}
{\prod_{k'\in V} \left(q_{\nu_{k'}}(\lambda^{i+})-q_{\nu_{k'}}(\lambda)\right)}.
\end{align*}
Then the claim follows from Theorems 
~\ref{th:diffq1}
and \ref{th:aibi}.
\end{proof}

\bigskip


\section{Proofs of Theorems} 

Instead of proving Theorem \ref{th:polynomial:noq}, we prove the following
more general result, which 
implies Theorem \ref{th:polynomial:noq}
when $\nu=\emptyset$.

\begin{thm} \label{th:contentmain}
Suppose that $\nu=(\nu_1, \nu_2, \ldots, \nu_\ell)$ is a given partition,  $\mu$ is a given strict partition  and $Q$ is a symmetric
function. Then there exists some $r\in \mathbb{N}$ such that
$$
D^r\left( \frac{Q\left(\binom{c_{\square}}{2}:
{\square}\in\lambda\right)q_{\nu}(\lambda)}{H_\lambda} \right)=0
$$
for every strict partition $\lambda$.
Consequently,
\begin{align*}
P(n)=\sum_{|\lambda/\mu|=n}\frac{2^{|\lambda|-|\mu|-\ell(\lambda)+\ell(\mu)}f_{\lambda/\mu}}{H_{\lambda}}Q\left(\binom{c_{\square}}{2}:
{\square}\in\lambda\right)q_{\nu}(\lambda)
\end{align*}
is a polynomial in $n$.
\end{thm}
\begin{proof}
 By linearity we can assume that
$$Q\left(\binom{c_{\square}}{2}:
{\square}\in\lambda\right)=\prod_{t=1}^s\sum_{\square\in
\lambda}{\binom{c_{\square}}{2}}^{r_t}$$ 
for some tuple $(r_1, r_2, \ldots, r_s)$.
Let
\begin{align*}
	A&= q_\nu(\lambda),\\
	\Delta_i A&= q_\nu(\lambda^{i+}) -q_\nu(\lambda),\\
	B&= 
\prod_{t=1}^s\sum_{\square\in \lambda}{\binom{c_{\square}}{2}}^{r_t},\\
	\Delta_i B
&= 
\prod_{t=1}^s\sum_{\square\in \lambda^{i+}}{\binom{c_{\square}}{2}}^{r_t}-
\prod_{t=1}^s\sum_{\square\in \lambda}{\binom{c_{\square}}{2}}^{r_t}.
\end{align*}
We have
\begin{align*}
	\Delta_i A
 &= \sum\limits_{(*)}\prod_{k\in U}q_{\nu_k}(\lambda)
{\prod_{k'\in V} \left(q_{\nu_{k'}}(\lambda^{i+})-q_{\nu_{k'}}(\lambda)\right)},\\
	\Delta_i B
&= \sum\limits_{(**)}\prod_{t\in U}\sum_{\square\in \lambda}{\binom{c_{\square}}{2}}^{r_t}
{\prod_{t'\in V} \left(
\sum_{\square\in \lambda^{i+}}{\binom{c_{\square}}{2}}^{r_{t'}}-
\sum_{\square\in \lambda}{\binom{c_{\square}}{2}}^{r_{t'}}
\right)}\\
&= \sum\limits_{(**)}\prod_{t\in U}\sum_{\square\in \lambda}{\binom{c_{\square}}{2}}^{r_t}
{\prod_{t'\in V} 
{\binom{x_i}{2}}^{r_{t'}}
},
\end{align*}
where the sum $(*)$ (resp. $(**)$) ranges over all
pairs $(U,V)$ of positive integer sets such that 
$U\cup V=\{1,2,\ldots,\ell\}$
(resp. $U\cup V=\{1,2,\ldots,s\}$), 
$U\cap V=\emptyset$ and $V\neq \emptyset$.
It follows from 
Corollary \ref{th:sumhlambda*} 
and 
Theorem \ref{th:Hlambda}
that
\begin{align*}
 \ &\ \ \ \ \ H_\lambda D\left( \frac{q_{\nu}(\lambda)\prod_{t=1}^s\sum_{\square\in
\lambda}{\binom{c_{\square}}{2}}^{r_t}}{H_\lambda}
\right)
\\
&=
\sum\limits_{\ell(\lambda^+)>\ell(\lambda)}
\frac{H_{\lambda}}{H_{\lambda^{+}}} \left(q_{\nu}(\lambda^+)\prod_{t=1}^s\sum_{\square\in
\lambda^+}{\binom{c_{\square}}{2}}^{r_t}-
q_{\nu}(\lambda)\prod_{t=1}^s\sum_{\square\in
\lambda}{\binom{c_{\square}}{2}}^{r_t}\right)
\\
&
\ \ \
+2\sum\limits_{\ell(\lambda^+)=\ell(\lambda)}\frac{H_{\lambda}}{H_{{\lambda^+}}}\left(q_{\nu}(\lambda^+)\prod_{t=1}^s\sum_{\square\in
\lambda^+}{\binom{c_{\square}}{2}}^{r_t}-
q_{\nu}(\lambda)\prod_{t=1}^s\sum_{\square\in
\lambda}{\binom{c_{\square}}{2}}^{r_t}\right)
\\
&=
\sum_{ i=0}^m \frac{\prod\limits_{\substack{1\leq
j\leq m}}\left({\binom{x_i}{2}}-{\binom{y_j}{2}}\right)}{\prod\limits_{\substack{0\leq j\leq m\\
j\neq i}}\left({\binom{x_i}{2}}-{\binom{x_j}{2}}\right)}\left(q_{\nu}(\lambda^{i+})\prod_{t=1}^s\sum_{\square\in
\lambda^{i+}}{\binom{c_{\square}}{2}}^{r_t}-
q_{\nu}(\lambda)\prod_{t=1}^s\sum_{\square\in
\lambda}{\binom{c_{\square}}{2}}^{r_t}\right)\\
&=
\sum_{ i=0}^m \frac{\prod\limits_{\substack{1\leq
j\leq m}}\left({\binom{x_i}{2}}-{\binom{y_j}{2}}\right)}{\prod\limits_{\substack{0\leq j\leq m\\
j\neq i}}\left({\binom{x_i}{2}}-{\binom{x_j}{2}}\right)}\
\left(
A\cdot \Delta_i B + B\cdot  \Delta_i A + \Delta_i A\cdot \Delta_i B
\right).
\end{align*}
Then by Theorems  \ref{th:aibi}, \ref{th:diffq1},
and \ref{th:diffq2}  each of the above three terms could be written
as a linear combination of some $\prod_{{\bar t}=1}^{\bar s}\sum_{\square\in
\lambda}{\binom{c_{\square}}{2}}^{{\bar r}_{\bar t}}q_{\bar \nu}(\lambda)$ 
satisfying one of the following two conditions:

(1) $\bar s<s$;

(2) $\bar s = s$ and $|\bar \nu| \leq  |\nu|-1$.

\noindent
Therefore the theorem follows by induction on $s$ and $|\nu|$.
\end{proof}

\begin{proof}[Proof of Theorem \ref{th:polynomial*}]
The special case
in the proof of Theorem \ref{th:contentmain}
with $\nu=\emptyset$ and $s=1$ yields
\begin{align*}
H_\lambda D\left( \frac{\sum_{\square\in
\lambda}{\binom{c_{\square}}{2}}^{r_1}}{H_\lambda}
\right)
&=
\sum_{0\leq i\leq m} \frac{\prod\limits_{\substack{1\leq
j\leq m}}\left({\binom{x_i}{2}}-{\binom{y_j}{2}}\right)}{\prod\limits_{\substack{0\leq j\leq m\\
j\neq i}}\left({\binom{x_i}{2}}-{\binom{x_j}{2}}\right)}\
\binom{x_i}{2}^{r_1}\\
& =\sum_{|\nu|\leq {r_1}}\xi_{\nu}q_\nu(\lambda),
\end{align*}
where $\xi_\nu$ are some constants.
The last equality is due to Theorem \ref{th:aibi}.
Notice that
$$(2k)!\binom{z+k-1}{2k}=2^k\prod_{i=1}^k\left(\binom{z}{2}-\binom{i}{2}\right).$$
Then by Theorems \ref{th:diffq2} and \ref{th:telescope} 
we know that
 \begin{align*}
P(n)=\sum_{|\lambda|=n}\frac{{f'}_{\lambda}}{H_{\lambda}}\sum_{\square\in\lambda}
\binom{c_{\square}+k-1}{2k}
\end{align*}
is a polynomial in $n$ with degree at most $k+1$. 

On the other hand,
$$
P(k+1)=\frac{{f'}_{(k+1)}}{H_{(k+1)}}\binom{2k}{2k}=\frac{2^k}{(k+1)!}
$$
since $\lambda=(k+1)$ is the only strict partition with size $k+1$ who has
contents  greater than~$k$. 
Moreover, it is obvious
that $P(0)=P(1)=\cdots=P(k)=0$. 
Since the polynomial $P(n)$ is uniquely determined by those values,
we obtain
$$P(n)=\frac{2^k}{(k+1)!}\binom{n}{k+1}.$$ 
\end{proof}

By Theorem \ref{th:contentmain}, the left-hand side of \eqref{eq:sumc:k=1} is a polynomial in~$n$. To evaluate this polynomial explicitly, 
 we need the following lemma. 
\begin{lem} \label{th:q1}
 Let $\lambda$ be a strict partition. Then $$q_1(\lambda)=|\lambda|.$$
\end{lem}
\begin{proof}
By the definition of the size of $\lambda$, we have
\begin{align*}
|\lambda|
=\sum_{i=1}^m\sum_{j=y_i}^{x_i-1}j
=\sum_{i=1}^{m}\left(\binom{x_i}{2}-\binom{y_i}{2}\right)=q_1(\lambda). 
\end{align*}
\end{proof}

\begin{proof}[Proof of Theorem \ref{th:k=1}]
It is easy to check that both sides of \eqref{eq:sumc:k=1} are equal for $n  =0, 1, 2$.
By Corollary \ref{th:sumhlambda*}, Theorem \ref{th:Hlambda} and Identity \eqref{eq:ab1} it is easy to see that 

\begin{align*}
 H_{\lambda}D\left(\frac{\sum_{\square\in\lambda }\binom{c_{\square}}{2}}{H_{\lambda}}\right)
 &=
 \sum\limits_{{\ell(\lambda^+)>\ell(\lambda)}}\frac{H_{\lambda}}{H_{{\lambda^{+}}}}\left(\sum_{\square\in\lambda^+ }\binom{c_{\square}}{2}-\sum_{\square\in\lambda }\binom{c_{\square}}{2}\right)
 \\
&\ \ \ +2\sum\limits_{{\ell(\lambda^+)=\ell(\lambda)}}\frac{H_{\lambda}}{H_{{\lambda^+}}}\left(\sum_{\square\in\lambda^+ }\binom{c_{\square}}{2}-\sum_{\square\in\lambda }\binom{c_{\square}}{2}\right)
\\
&=
\sum_{0\leq i\leq m} \frac{\prod\limits_{\substack{1\leq
j\leq m}}\left({\binom{x_i}{2}}-{\binom{y_j}{2}}\right)}{\prod\limits_{\substack{0\leq j\leq m\\
j\neq i}}\left({\binom{x_i}{2}}-{\binom{x_j}{2}}\right)}\binom{x_i}{2}
 \\&
 =q_1(\lambda)\\
 &
 =|\lambda|.
 \end{align*}
 Therefore we have
 \begin{align*}
H_{\lambda}D^2\left(\frac{\sum_{\square\in\lambda }\binom{c_{\square}}{2}}{H_{\lambda}}\right)
&=
1,
\\
H_{\lambda}D^3\left(\frac{\sum_{\square\in\lambda }\binom{c_{\square}}{2}}{H_{\lambda}}\right)
&=
0.
 \end{align*}
Then our claim follows from Theorem \ref{th:telescope}.
\end{proof}
\medskip
Similarly, by \eqref{eq:ab0}, \eqref{eq:ab1} and \eqref{eq:ab2} we have
 \begin{align*}
 H_{\lambda}D\left(\frac{\sum_{\square\in\lambda }\binom{c_{\square}+1}{4}}{H_{\lambda}}\right)
 &=
 \frac1{12}\left( q_2(\lambda)+|\lambda|^2-2|\lambda|\right),
 \\
H_{\lambda}D^2\left(\frac{\sum_{\square\in\lambda }\binom{c_{\square}+1}{4}}{H_{\lambda}}\right)
 &=
 \frac23|\lambda|,
 \\
 H_{\lambda}D^3\left(\frac{\sum_{\square\in\lambda }\binom{c_{\square}+1}{4}}{H_{\lambda}}\right)
 &=
 \frac23,
 \\
 H_{\lambda}D^4\left(\frac{\sum_{\square\in\lambda }\binom{c_{\square}+1}{4}}{H_{\lambda}}\right)
 &=
 0.
\end{align*}
Thus by Theorem \ref{th:telescope} we obtain the following result.
\begin{thm} \label{th:k=2}
Let $\mu$ be a strict partition. Then 
  \begin{align*}
& \sum_{|\lambda/\mu|=n}\frac{2^{|\lambda|-\ell(\lambda)-|\mu|+\ell(\mu)}f_{\lambda/\mu}H_\mu}{H_{\lambda}}\left(\sum_{\square\in\lambda}\binom{c_{\square}+1}{4}-\sum_{\square\in\mu}
\binom{c_{\square}+1}{4}\right)
\\
=\ & \frac23\binom{n}{3}+\frac23|\mu|\binom{n}{2}+\frac1{12}\left( q_2(\mu)+|\mu|^2-2|\mu|\right)n.
\end{align*}
\end{thm}

\bigskip

\section{Acknowledgments}
The authors thank the referees for carefully reading the paper and for giving helpful comments.
The second author  thanks  Prof. P.-O. Dehaye for the encouragements and helpful suggestions. 
The second author is supported  by Grant [PP00P2\_138906] of the Swiss National Science Foundation.




\end{document}